\documentclass[a4paper,article,11pt]{article}
\usepackage{amssymb}
\usepackage{amstext}
\usepackage{amsmath}
\usepackage{amscd}
\usepackage{amsthm}
\usepackage{amsfonts}
\usepackage{enumerate}
\usepackage{latexsym}
\usepackage{mathrsfs}
\usepackage{mathtools}
\usepackage{authblk}

\usepackage[dvipdfmx]{graphicx}

\makeatletter

\@addtoreset{equation}{section}
\makeatother

\newtheorem{theorem}{Theorem}[section]

\newtheorem{proposition}[theorem]{Proposition}
\newtheorem{remark}[theorem]{Remark}

\DeclareMathOperator{\wn}{Wind}
\DeclareMathOperator{\Ran}{Ran}

\DeclareMathOperator{\supp}{supp}

\DeclareMathOperator{\Index}{Index}

\DeclareMathOperator{\ac}{ac}

\DeclareMathOperator{\pp}{p}

\newcommand{\strong}{\mathop{\rm s\mathchar`-lim}\limits}

\def\C{\mathbb C}

\def\S{\mathbb S}
\def\R{\mathbb R}
\def\RR{\mathscr R}
\def\Z{\mathbb Z}

\def\B{\mathcal B}
\def\d{\mathrm d}

\def\e{\mathrm e}
\def\E{\mathcal E}
\def\F{\mathscr F}
\def\H{\mathcal H}

\def\K{\mathcal K}

\def\T{\mathbb T}
\def\TT{\mathcal T}

\def\P{\mathcal P}

\def\one{\mathbf 1}

\def\g{\mathfrak{g}}

\def\f{\mathfrak{f}}
\def\PV{\mathcal{P}}

\newcommand{\xrightarrowdbl}[2][]{%
  \xrightarrow[#1]{#2}\mathrel{\mkern-14mu}\rightarrow}

\begin{document}

\title{Schr\"odinger wave operators on the discrete half-line} 
 
\author[1]{Hideki Inoue}
\author[2]{Naohiro Tsuzu}
\affil[1]{Graduate school of mathematics, Nagoya university,}
\affil[2]{Department of mathematics, Nagoya University,}
\affil[ ]{Chikusa-ku, Nagoya 464-8602, Japan}
\affil[ ]{\textit {m16007v@math.nagoya-u.ac.jp, tsuzu.naohiro@j.mbox.nagoya-u.ac.jp}}

\date{}

\maketitle

\begin{abstract}
An explicit formula for the wave operators associated with Schr\"odinger operators on the discrete half-line is deduced from their stationary expressions. The formula enables us to understand the wave operators as one dimensional pseudo-differential operators of order zero. As an application, we give a topological interpretation for Levinson's theorem, which relates the scattering phase shift and the number of bound states of the system.
\end{abstract}

\smallskip


\section{Introduction}\label{Introduction}

In this paper, we investigate scattering theory for discrete Schr\"odinger operators on the half-line $\Z_+=\{0,1,2,\ldots\}$ and focus on its topological aspect. More precisely, we study the self-adjoint operator $H=H_0+V(X)$, where $H_0=(T+T^*)/2$ with $T$ the shift operator on $\ell^2(\Z_+)$ defined by $T\delta_n=\delta_{n+1}$ for the canonical orthonormal basis $\{\delta_n\}_{n\in\Z_+}$, and where $V(X)$ is the multiplication operator by a function $V:\Z_+\to\R$. Note that the free operator $H_0$ is purely absolutely continuous and $\sigma(H_0)=[-1,1]$. Throughout this paper we assume $V$ to satisfy:
\begin{equation}\label{assumption}
\exists\rho>5/2\qquad\text{s.t.}\qquad \sup_{n\in\Z_+}|(1+n)^{\rho}V(n)|<\infty.
\end{equation}

Since the perturbation $V(X)$ is trace class under the assumption \eqref{assumption}, it follows from the Kato-Rosenblum theorem that the wave operators $W_{\pm}:=\strong_{t\to\pm\infty}\e^{itH}\e^{-itH_0}$ exist and are complete, i.e the ranges $\Ran(W_{\pm})$ coincide with the absolutely continuous subspace $\H_{\ac}(H)$ of $H$. As a consequence, the scattering operator ${\rm\bf S}:=W_+^*W_-$ is a unitary operator on $\ell^2(\Z_+)$ commuting with $H_0$. Our main result is an explicit formula for $W_-$.


\begin{theorem}\label{main result 1}
Under the assumption \eqref{assumption}, the equality
\begin{equation}\label{formula of wo}
W_-=\one+\frac{1}{2}\left\{\one-\tanh(\pi A)+i\tanh(B/2)\cosh(\pi A)^{-1}\right\}({\rm\bf S}-\one)+K
\end{equation}
holds with $K$ a compact operator on $\ell^2(\Z_+)$. Here, $B:=\tanh^{-1}(H_0)$ and $A$ are self-adjoint operators in $\ell^2(\Z_+)$ satisfying the Weyl relation
\begin{equation}\label{Weyl CCR}
\e^{itB}\e^{isA}=\e^{-ist}\e^{isA}\e^{itB},\qquad \forall s,t\in\R.
\end{equation}
\end{theorem}


To simplify our presentation, we only consider $W_-$ but a similar formula for $W_+$ can be obtained by using the relation $W_+=W_-{\rm\bf S}^*$ for instance.


In the last 10 years, such explicit formulas for wave operators have been obtained for example in \cite{I,KR} for continuous systems (see also the review paper \cite{R}) and in \cite{BS} for a discrete system. We call $B$ the rescaled energy operator. The operator $A$ can be seen as the operator canonically conjugate to $B$ by the relation \eqref{Weyl CCR}. Note that for continuous systems, where the free operator is the Laplacian $-\Delta$, the rescaled energy operator is often $\ln\left(\sqrt{-\Delta}\right)$ and $A$ is given by the generator of dilations.


In contrast to continuous systems, there is no canonical definition of the dilation group for discrete systems. For tight-binding models  \cite{BS} in $\Z^{d}$ with $d\geq 3$, $A$ is constructed from a classical energy gradient flow. In the present situation, since the spectrum of $H_0$ is $[-1,1]$ and simple, we define $A:=\RR^* (-i\partial_x)\RR$ with a unitary operator $\RR:\ell^2(\Z_+)\to L^2(\R)$ based on the diffeomorphism $\tanh :\R\to (-1,1)$. We call the triple $(\RR,B,A)$ a \emph{rescaled energy representation} for $(\ell^2(\Z_+),H_0)$.


We also give a topological interpretation for \emph{Levinson's theorem}, which is a relation between the scattering phase shift and the number of bound states of the perturbed system $H$ (see the equality \eqref{classical Levinson}), based on the index theorem approach of \cite{KR}. Let us mention that Levinson type theorems for discrete systems have been studied for example in \cite{CS,CG,HKS,Y2}. The formula \eqref{formula of wo} is used to prove the affiliation of $W_-$ to a $C^*$-algebra $\E_{\square}$ represented in $\ell^2(\Z_+)$, which has been a major algebraic tool for the topological approach to Levinson's theorem \cite{I,KR,R}.


Let us now explain the contents of this paper. In Section 2, borrowing the relevant material from the paper \cite{Y2}, we review the spectral analysis and the direct approach to scattering theory for the perturbed operator $H$. In Section 3, we deduce an expression for $W_-$ from its stationary expression in terms of the generalized Fourier transform $\F_{-}:\ell^2(\Z_+)\to L^2(-1,1)$ for $H$. At this point, the leading term is expressed in terms of the scattering operator ${\rm\bf S}$ and a bounded operator $U$, which is a product of discrete Fourier (co)sine type transforms and is independent of the potential $V$.  It has been observed that a similar product of Fourier (co)sine transforms appears for a continuous system and is expressible in terms of the generator of dilations (see \cite{I}). In Section 4, an expression of the operator $U$ in a rescaled energy representation $(\RR,B,A)$ for $(\ell^2(\Z_+),H_0)$ is established, and the proof of Theorem \ref{main result 1} is given. In Section 5, we discuss a topological version of Levinson's theorem for $H$. An expression for $T$ in $(\RR,B,A)$ is also obtained, and a relation between the Toeplitz algebra and $\E_{\square}$ is discussed in Remark \ref{remark3}. 


As a final remark, we confess that our assumption on the model is not the optimal one. Indeed, in \cite{Y2} Levinson's theorem is established for Jacobi operators of the form $(Ta(X)+a(X)T^*)/2+b(X)$, where $a:\Z_+\to(0,\infty)$ and $b:\Z_+\to\R$ satisfy that in the power scale $|a(n)-1|+|b(n)|=O(n^{-\rho})$ as $n\to\infty$ for some $\rho>2$. Furthermore, in \cite{I} an explicit formula for the wave operators associated with Schr\"odinger operators on the continuous half-line is established for potentials decaying faster than $x^{-2}$ as $x\to\infty$. These assumptions are optimal in the sense that corresponding systems possess only finitely many bound states. To simplify our presentation we concentrate on the discrete Schr\"odinger operator $H$ with $V$ satisfying \eqref{assumption}, but we hope that our paper will encourage future works on more complicated models.

\section{Preliminaries}\label{Preliminaries}

In this section, we briefly recall the spectral analysis of the operator $H=H_0+V(X)$ based on the paper \cite{Y2}. As in the continuous models on the half-line, the key role is played by two distinguished solutions, called the \emph{regular solution} $\varphi(\cdot,z)=\{\varphi(n,z)\}_{n=-1}^{\infty}$ and the \emph{Jost solution} $\vartheta(\cdot,z)=\{\vartheta(n,z)\}_{n=-1}^{\infty}$, of the Schr\"odinger equation
\begin{equation}\label{Schroedinger}
\frac{u(n-1)+u(n+1)}{2}+V(n)u(n)=zu(n),\qquad n\in\Z_+
\end{equation}
for $z\in\C\setminus[-1,1]=:\Pi$. Note that in \cite{Y2} $\varphi(n,z)$ and $\vartheta(n,z)$ are denoted by $P_n(z)$ and $f_n(z)$, respectively, but we prefer notations which are similar to those used in the continuous model \cite{I,Y1}. These two solutions are distinguished by the following boundary conditions:
\begin{equation}\label{regular boundary}
\varphi(-1,z)=0,\ \varphi(0,z)=1
\end{equation}
and $\vartheta(n,z)=\zeta(z)^n\left(1+o(1)\right)$ as $n\to\infty$,
where $\zeta(z)=z-\sqrt{z^2-1}$ and the branch of  analytic function $\sqrt{z^2-1}$ of $z\in\Pi$ is fixed by the condition $\sqrt{z^2-1}>0$ if $z>1$. Note that the function $\sqrt{z^2-1}$ continuously extends to the closure of $\Pi$, equals $\pm i\sqrt{1-\lambda^2}$ for $z=\lambda\pm i0$ with $\lambda\in(-1,1)$ and  that $\sqrt{z^2-1}<0$ for $z<-1$. In the following,  we write $\zeta(\lambda)$ for $\zeta(\lambda+i0)$ for $\lambda\in (-1,1)$. We also note that $|\zeta(\lambda+i\varepsilon)|\leq1$ for any $\lambda\in\R$ and $\varepsilon\geq0$, and that the equality holds iff $\lambda\in[-1,1]$ and $\varepsilon=0$.

One can easily infer the existence of the regular solution $\varphi(\cdot,z)$ by viewing the equation \eqref{Schroedinger} satisfying \eqref{regular boundary} as a recurrent equation. In the free case, when $V\equiv 0$, the regular solution $\varphi_0(\cdot,z)$ is given by the formula
\begin{equation*}
\varphi_0(n,z)=\frac{1}{2\sqrt{z^2-1}}\left(\zeta(z)^{-n-1}-\zeta(z)^{n+1}\right).
\end{equation*}
In particular, $\varphi_0(n,z)=\sin\left((n+1)\theta\right)/\sqrt{1-\lambda^2}$ for $\lambda=\cos(\theta)\in(-1,1)$ with $\theta\in(0,\pi)$ and $\varphi_0(n,\cdot)$ is the $n$-th Chebyshev polynomial of the second kind. 

The Jost solution $\vartheta(\cdot,z)$ is defined for $z$ in the closure of $\Pi$ with $z\neq\pm1$, and can be characterized as the unique solution of the Volterra integral equation
\begin{equation}\label{Volterra for Jost}
\vartheta(n,z)=\vartheta_0(n,z)-\frac{1}{\sqrt{z^2-1}}\sum_{m=n+1}^{\infty}\Bigl(\zeta(z)^{n-m}-\zeta(z)^{m-n}\Bigr)V(m)\vartheta(m,z),\qquad n\in\Z_+,
\end{equation}
where $\vartheta_0(n,z):=\zeta(z)^{n}$ is the Jost solution in the free case. Then, for each $n\in\Z_+$ the function $z\mapsto \vartheta(n,z)$ is 
analytic in $\Pi$ and continuous up to the cut along $[-1,1]$.  If we set $\vartheta(n,\lambda):=\vartheta(n,\lambda+i0)$ for $\lambda=\cos(\theta)$, $\theta\in(0,\pi)$, then the corresponding integral equation is
\begin{equation}\label{Volterra for real}
\vartheta(n,\lambda)=\vartheta_0(n,\lambda)-2\sum_{m=n+1}^{\infty}\frac{\sin\bigl((m-n)\theta\bigr)}{\sin(\theta)}V(m)\vartheta(m,\lambda),\qquad n\in\Z_+.
\end{equation}


By using the fact that $|\sin(Nt)/\sin(t)|\leq N$ for any $N\in\Z_+$ and $t\in\R$, the following estimate can be proved by the same method as in the continuous case \cite[Chap. 4, Lem. 3.1]{Y1} :
\begin{equation}\label{estimate0}
|\vartheta(n,\lambda)-\vartheta_0(n,\lambda)|\leq \exp\left(\sum_{m=n+1}^{\infty}(m-n)|V(m)|\right)-1,\qquad \lambda\in[-1,1], n\in\Z_+
\end{equation}
(see also \cite[Appendix A.1]{Y2} for the method of iterations). In particular, the following form of the inequality \eqref{estimate0}, which can easily be deduced from the mean-value theorem, is useful for our purpose: there exists $C>0$ such that
\begin{equation}\label{estimate1}
|\vartheta(n,\lambda)-\vartheta_0(n,\lambda)|\leq C (1+n)^{-\rho+2},\qquad n\in\Z_+,\lambda\in[-1,1].
\end{equation}
Note also that $\vartheta(\cdot,\pm 1)=\{\vartheta(n,\pm 1)\}_{n=-1}^{\infty}$ is a solution of \eqref{Schroedinger} for $z=\pm 1$ and $\vartheta(n,\pm1)=(\pm 1)^{n}+o(1)$ as $n\to\infty$ (see \cite[Thm. 2.1 \& 4.1]{Y2}).

For two solutions $u$ and $v$ of  \eqref{Schroedinger}, their Wronskian is defined by 
\begin{equation}
\{u,v\}:=2^{-1}\bigl(u(n)v(n+1)-u(n+1)v(n)\bigr),
\end{equation}
and one can easily see that it does not depend on $n\in\Z_+$. Hence, it follows from the boundary condition \eqref{regular boundary} that $\omega(z):=\{\varphi(\cdot,z),\vartheta(\cdot,z)\}=-2^{-1}\vartheta(-1,z)$ and $\omega_0(z):=\{\varphi_0(\cdot,z),\vartheta_0(\cdot,z)\}=-2^{-1}\zeta(z)^{-1}$. Then, \emph{the Jost function} is defined by
$$
\Omega(z):=\omega(z)/\omega_0(z)=-2\zeta(z)\omega(z)=\zeta(z)\vartheta(-1,z).
$$
As for potential scattering on the continuous half-line, for $z\in\Pi$, $\Omega(z)=0$ if and only if $z$ is an eigenvalue of $H$ with the eigenfunction $\vartheta(\cdot,z)_{\restriction_{\Z_+}}$.  $\Omega(\lambda)\equiv\Omega(\lambda+i0)\neq 0$ also holds for $\lambda\in(-1,1)$ and the function $[-1,1]\ni\lambda\mapsto\Omega(\lambda)$ is continuous. It is therefore natural to say that $H$ has a  \emph{threshold resonance}  at $z= +1$ or $z=-1$ if $\Omega(+1)=0$ or $\Omega(-1)=0$, respectively. Note also that $\vartheta(\cdot,\lambda)$ and $\overline{\vartheta(\cdot,\lambda)}$ are linearly independent solutions of \eqref{Schroedinger} for $z=\lambda\in(-1,1)$ and the equality
\begin{equation}\label{regular by Jost}
\varphi(n,\lambda)=i\left(\omega(\lambda)\overline{\vartheta(n,\lambda)}-\overline{\omega(\lambda)}\vartheta(n,\lambda)\right)(1-\lambda^2)^{-1/2}
\end{equation}
holds for each $n\in\Z_+$.

 
 We finally introduce some quantities associated to $H$. The limit amplitude $a$ and the phase shift $\eta$ associated with $H$ are continuous functions on $(-1,1)$ defined by the relation
 $$
 \Omega(\lambda)=a(\lambda)\e^{i\eta(\lambda)}, \qquad a(\lambda)=|\Omega(\lambda)|.
 $$ 
We also define the \emph{scattering matrix} $s(\lambda):=\overline{\Omega(\lambda)}/\Omega(\lambda)=\e^{-2i\eta(\lambda)}$ for $\lambda\in(-1,1)$. Note that we defined $a$ and $\eta$ as functions of the spectral parameter $\lambda$ while they are defined as functions of $\theta$ in \cite{Y2}. 
Note that $\eta$ is determined only up to modulo $2\pi$. However, this is sufficient for our purpose since we are interested in an identity of the form
\begin{equation}\label{classical Levinson}
\eta(+1)-\eta(-1)=\pi\Bigl(N+\Delta_{-}+\Delta_{+}\Bigr),
\end{equation}
where $N:=\#\sigma_{\pp}(H)<\infty$ is the number of bound states of $H$ (see \cite[Rem. 4.12]{Y2}) and the correction terms $\Delta_{\pm}=0$ if $\Omega(\pm1)\neq 0$ and $\Delta_{\pm}=1/2$ if $\Omega(\pm1)=0$. The relation \eqref{classical Levinson} is called \emph{Levinson's theorem} and the usual proof based on complex analysis can be found in \cite[Thm. 2.2]{HKS} or \cite[Thm. 5.10]{Y2}. 

\begin{remark}\label{remark1}
For discrete systems, the asymptotic behaviors of  the Jost function near thresholds have been studied for example in \cite[Thm 2.1]{HKS} or \cite{Y2}. In particular, by using \cite[Lem. 4.8 \&Thm. 4.9]{Y2} one can easily deduce that $s(\pm1)=1$ if $\Omega(\pm1)\neq0$ and $s(\pm1)=-1$ if $\Omega(\pm1)=0$.
\end{remark}


\section{Stationary expressions}\label{analysis}

In this section, we deduce stationary expressions for the wave operators $W_{\pm}$ for the pair $(H,H_0)$ in terms of the \emph{generalized Fourier transforms} which have been constructed in \cite[Sec. 6]{Y2}.  In order to get formula \eqref{formula of wo}, we adopt the approach used in the continuous case \cite{I}, where the formula is obtained by rewriting the decomposition of the kernel of the generalized Fourier transform \cite[Chap. 4, eq.(2.33)]{Y1} as an operator identity.


We first define the wave functions
\begin{equation}
\psi_{\pm}(n,\lambda):=\sqrt{\frac{2}{\pi}}\frac{\varphi(n,\lambda)\sigma_{\mp}(\lambda)(1-\lambda^2)^{1/4}}{|\Omega(\lambda)|},\qquad n\in\Z_+,\lambda\in(-1,1),
\end{equation}
where $\sigma_+(\lambda):=\Omega(\lambda)/|\Omega(\lambda)|$ and $\sigma_-(\lambda)=\overline{\sigma_+(\lambda)}$. They satisfy $\sigma_{+}(\lambda)\sigma_-(\lambda)=1$ and $\sigma_-(\lambda)^2=s(\lambda)$. In the free case, since the corresponding Jost function identically equals $1$, we have
$$
\psi_{\pm,0}(n,\lambda)\equiv\psi_{\sin}(n,\lambda):=\sqrt{\frac{2}{\pi}}\frac{\sin\bigl((n+1)\theta\bigr)}{(1-\lambda^2)^{1/4}}, \qquad \lambda=\cos(\theta),\ \theta\in(0,\pi).
$$

The generalized Fourier transforms $\F_{\pm}$ associated with $H$ are defined by 
\begin{equation}\label{GFT}
[\F_{\pm}u](\lambda):=\sum_{n\in\Z_+}\psi_{\pm}(n,\lambda)u(n),\qquad \lambda\in(-1,1)
\end{equation}
for $u\in C_{\rm c}(\Z_+)$. Note that the operators $\F_{\pm}$ differ from the corresponding operator $F$ used in \cite[eq.(6.6)]{Y2} by the multiplicative factors $\sigma_{\mp}(\lambda)$. Then, as shown in \cite[Sec. 6]{Y2}, one can prove that $\F_{\pm}$ continuously extend to co-isometries from $\ell^2(\Z_+)$ to $L^2(-1,1)$ with $\Ran(\F_{\pm}^*)=\H_{\ac}(H)$. The adjoints $\F_{\pm}^*$ are given by $[\F_{\pm}^*g](n)=\int_{-1}^1 \psi_{\mp}(n,\lambda)g(\lambda)\d\lambda$ for $g\in L^2(-1,1)$ and $n\in\Z_+$. Moreover, $\F_{\pm}$ satisfy the intertwining relations $\F_{\pm}H=L\F_{\pm}$, where $L$ is the multiplication operator on $L^2(-1.1)$ by the function $\lambda\mapsto\lambda$.


In the free case, $\F_-=\F_+$ and they are simply denoted by $\F_{\sin}$. $\F_{\sin}$ is a unitary operator satisfying $\F_{\sin}H_0=L\F_{\sin}$ since $\{\psi_{\sin}(n,\cdot)\}_{n\in\Z_+}$ forms a complete orthonormal basis of $L^2(-1,1)$. We also define an auxiliary operator $\F_{\cos}$ by replacing the kernel $\psi_{\pm}(n,\lambda)$ in \eqref{GFT} with the kernel
\begin{equation*}
\psi_{\cos}(n,\lambda):=\sqrt{\frac{2}{\pi}}\frac{\cos\bigl((n+1)\theta\bigr)}{(1-\lambda^2)^{1/4}}, \qquad \lambda=\cos(\theta),\ \theta\in(0,\pi).
\end{equation*}

 
Now, we recall some basic results \cite{Y2} on the stationary scattering theory for $H$. Note that the following theorem can be applied to a more general class of perturbations of $H_0$, see \cite[Thm 6.4, 6.5 \& 6.7]{Y2}.
\begin{theorem}\label{theorem 3.1}
Under the assumption \ref{assumption}, the wave operators $W_{\pm}$ have stationary expressions $W_{\pm}=\F_{\pm}^*\F_{\sin}$, and accordingly they are complete. Moreover, the scattering operator ${\rm\bf S}:=W_+^*W_-$ has an expression ${\rm\bf S}=\F_{\sin}^*s\F_{\sin}$, where $s$ denotes the multiplication operator by the function $\lambda\mapsto s(\lambda)$.
\end{theorem}


We now deduce an explicit formula for $W_-$ from the above stationary expression by looking at the kernel $\psi_+(n,\lambda)=\overline{\psi_-(n,\lambda)}$ of $\F_-^*$. By using \eqref{regular by Jost}, we have
\begin{align*}
\psi_+(n,\lambda)&=\sqrt{\frac{2}{\pi}}\frac{\overline{\vartheta(n,\lambda)\zeta(\lambda)}-s(\lambda)\vartheta(n,\lambda)\zeta(\lambda)}{2i(1-\lambda^2)^{1/4}}\\
&=\psi_{\sin}(n,\lambda)+\sqrt{\frac{2}{\pi}}\frac{i\zeta(\lambda)^{n+1}}{2(1-\lambda^2)^{1/4}}\bigl(s(\lambda)-1\bigr)+K_0(n,\lambda)\notag\\
&=\psi_{\sin}(n,\lambda)+\frac{i\psi_{\cos}(n,\lambda)+\psi_{\sin}(n,\lambda)}{2}\bigl(s(\lambda)-1\bigr)+K_0(n,\lambda),
\end{align*}
where
\begin{equation}\label{remainder term}
K_0(n,\lambda):=\sqrt{\frac{2}{\pi}}\frac{\overline{p(n,\lambda)\zeta(\lambda)}-s(\lambda)p(n,\lambda)\zeta(\lambda)}{2i}\qquad\text{with}\qquad p(n,\lambda):=\frac{\vartheta(n,\lambda)-\vartheta_0(n,\lambda)}{(1-\lambda^2)^{1/4}}.
\end{equation}
In the sequel, we use the same letters for integral operators and their kernels. Then, based on Theorem \ref{theorem 3.1} and the above expansion, one infers
\begin{equation}\label{formula1}
W_-={\bf 1}+\frac{U+{\bf 1}}{2}\bigl({\rm\bf S}-{\bf 1})+K_0\F_{\sin},
\end{equation}
where $U:=i\F_{\cos}^{*}\F_{\sin}$.


 \section{Rescaled energy representation and the proof of Theorem \ref{main result 1}}\label{topology}

The aim of this section is to understand the operator $U=i\F_{\cos}^*\F_{\sin}$ in a rescaled energy representation $(\RR,B,A)$ defined below. This new representation enables us to consider $U$, and therefore $W_-$ as pseudo-differential operators on $L^2(\R)$. Note that the expression for the operator $(U+1)/2$ in \eqref{formula1} in front of ${\rm\bf S}-\one$ is universal in the sense that it is independent of the potential, and such a phenomenon has been commonly observed for continuous systems (see \cite{R} for more information).


We introduce a unitary operator $\RR_{0}:L^2(-1,1)\to L^2(\R)$ defined by
\begin{equation}
[\RR_0 g](\beta):=\cosh(\beta)^{-1}g\left(\tanh(\beta)\right),\qquad \text{a.e.}\ \beta\in\R
\end{equation}
for $g\in L^2(-1,1)$. Then, a unitary operator $\RR:\ell^2(\Z_+)\to L^2(\R)$ is defined by $\RR:=\RR_0\F_{\sin}$. One observes that the self-adjoint operator $B$ defined in Theorem \ref{main result 1} satisfies $\tanh(B)=H_0$. The operator $\RR B\RR^*$ corresponds to the multiplication operator $X$ on $L^2(\R)$ by the function $\beta\mapsto\beta$, and the Weyl relation \eqref{Weyl CCR} is satisfied for $B$ and $A:=\RR^* D\RR$ with $D=-i\partial_{\beta}$ the momentum operator in $L^2(\R)$.


\begin{proposition}\label{prop1}
The following equality holds:
\begin{equation}\label{expression for U}
U=-\tanh(\pi A)+i\tanh(B/2)\cosh(\pi A)^{-1}+(\text{a compact operator}).
\end{equation}
\end{proposition}
\begin{proof}
For $f,g\in C_{c}^{\infty}(-1,1)$, we have formally
\begin{align*}
&\left\langle \F_{\sin}^*f,\left(\frac{U+\one}{2}\right)\F_{\sin}^*g\right\rangle_{\ell^2(\Z_+)}\\
&=\sum_{n=0}^{\infty }\overline{[\F_{\sin}^*f](n)}\left(\frac{i}{2}\sqrt{\frac{2}{\pi}}\int_{-1}^1\lim_{\varepsilon\downarrow 0}\frac{\zeta(\nu+i\varepsilon)^{n+1}}{(1-\nu^2)^{1/4}}g(\nu)\d \nu\right)\\
&=\lim_{\varepsilon\downarrow 0}\sum_{n=0}^{\infty }\overline{[\F_{\sin}^*f](n)}\left(\frac{i}{2}\sqrt{\frac{2}{\pi}}\int_{-1}^1 \frac{\zeta(\nu+i\varepsilon)^{n+1}}{(1-\nu^2)^{1/4}}g(\nu)\d \nu\right)\\
&=\lim_{\varepsilon\downarrow 0}\int_{-1}^{1}\overline{f(\lambda)}\left[\sum_{n=0}^{\infty }\psi_{\sin}(n,\lambda)\left(\frac{i}{2}\sqrt{\frac{2}{\pi}}\int_{-1}^1 \frac{\zeta(\nu+i\varepsilon)^{n+1}}{(1-\nu^2)^{1/4}}g(\nu)\d \nu\right)\right]\d\lambda.
\end{align*}
The above computations can be justified as follows: note first that the series $\sum_{n=0}^{\infty}[\F_{\sin}^*f](n)$ converges absolutely since $f\in C_{\rm c}^{\infty}(-1,1)$. By using the fact that $\sup_{\nu\in\supp g}|\zeta(\nu+i\varepsilon)|\leq 1$ for any $\varepsilon\geq 0$, where the equality holds iff $\varepsilon=0$, we justify the step from the second to the third line by applying the Lebesgue dominated convergence theorem twice. Finally, we obtain the last expression by the definition of the operator $\F_{\sin}$.


By using Fubini's theorem, one has for each fixed $\varepsilon>0$
\begin{align*}
&\sum_{n=0}^{\infty}\psi_{\sin}(n,\lambda)\frac{i}{2}\sqrt{\frac{\pi}{2}}\int_{-1}^1 \frac{\zeta(\nu+i\varepsilon)^{n+1}}{(1-\nu^2)^{1/4}}g(\nu)\d \nu\\
&=\frac{i}{2}\sqrt{\frac{\pi}{2}}\int_{-1}^1\sum_{n=0}^{\infty}\psi_{\sin}(n,\lambda)\frac{\zeta(\nu+i\varepsilon)^{n+1}}{(1-\nu^2)^{1/4}}g(\nu)\d \nu\\
&=\frac{i}{\pi}\int_{-1}^{1}(1-\lambda^2)^{1/4}\sum_{n=0}^{\infty}\frac{\sin\left((n+1)\theta\right)}{\sin\theta}\frac{\zeta(\nu+i\varepsilon)^{n+1}}{(1-\nu^2)^{1/4}}g(\nu)\d\nu,
\end{align*}
where $\lambda=\cos(\theta)$. Now, by using the formula for the generating function of the Chebyshev polynomials of the second kind (see \cite[18.3.4.1]{J}), one gets
\begin{equation}
\left\langle f,\F_{\sin}\left(\frac{U+\one}{2}\right)\F_{\sin}^*g\right\rangle_{L^2(-1,1)}=\lim_{\varepsilon\downarrow 0}\langle f,L_{\varepsilon}g\rangle_{L^2(-1,1)},
\end{equation}
where the kernel of $L_{\varepsilon}$  is $L_{\varepsilon}(\lambda,\nu):=(i/2\pi)(1-\lambda^2)^{1/4}(\nu-\lambda+i\varepsilon)^{-1}(1-\nu^2)^{-1/4}$.


For the next step, we set $\f(\lambda):=(1-\lambda^2)^{1/4}\overline{f(-\lambda)}$ and $\g({\nu}):=(1-\nu^{2})^{-1/4}g(\nu)$, and consider $\f$, $\g$ and the convolution $\f\ast\g$ as elements of $C_{\rm c}^{\infty}(\R)$. Then, by Fubini's theorem and a change of variables
\begin{align*}
\langle f,L_{\varepsilon}g\rangle_{L^2(-1,1)}
&=\frac{i}{2\pi}\int_{\R}\frac{(\f\ast\g)(\nu)}{\nu+i\varepsilon}\d\nu\to\frac{(\f\ast\g)(0)}{2}+\frac{i}{2\pi}\PV\int_{\R}\frac{(\f\ast\g)(\nu)}{\nu}\d\nu,\qquad \text{as}\ \varepsilon\downarrow 0.
\end{align*}
Here, the symbol $\PV$ stands for the principal value integral and we have used the Sokhotski-Plemelj formula. By using the strong $L^2(\R)$ convergence of the truncated Hilbert transform (see \cite[Chap. 3, Sec. 4]{D}), we obtain
\begin{equation}\label{expression of U as singular integral}
\left\langle f,\F_{\sin}\left(\frac{U+\one}{2}\right)\F_{\sin}^*g\right\rangle_{L^2(-1,1)}=\frac{1}{2}\langle f,g\rangle_{L^2(-1,1)}+\left\langle f,\PV L_0g \right\rangle_{L^2(-1,1)},
\end{equation}
where the operator $\PV L_0$ is a singular integral operator with the Schwartz kernel $L_0(\lambda,\nu):=(i/2\pi)(1-\lambda^2)^{1/4}(\nu-\lambda)^{-1}(1-\nu^2)^{-1/4}$.


By a change of variable, one can easily observe that for $h\in C_{\rm c}^{\infty}(\R)$ and a.e. $\beta\in\R$
\begin{align*}
\left[\RR_0(\PV L_0) \RR_0^* h\right](\beta)&=\frac{i}{2\pi}\PV \int_{\R}\cosh(\beta)^{-1/2}\frac{h(\gamma)}{\sinh(\gamma-\beta)}\cosh(\gamma)^{1/2}\d\gamma\\
&=\frac{i}{2\pi}\PV \int_{\R}\frac{b(\beta)}{\e^{\beta/2}+\e^{-\beta/2}}\frac{h(\gamma)}{\sinh(\gamma-\beta)}\frac{\e^{\gamma/2}+\e^{-\gamma/2}}{b(\gamma)}\d\gamma\\
&=\frac{i}{4\pi}\PV \int_{\R}\frac{b(\beta)\, \e^{\beta/2}}{\e^{\beta/2}+\e^{-\beta/2}}\left[\frac{1}{\sinh\left((\gamma-\beta)/2\right)}+\frac{1}{\cosh\left((\gamma-\beta)/2\right)}\right]\frac{h(\gamma)}{b(\gamma)}\d\gamma\\
&+\frac{i}{4\pi}\PV \int_{\R}\frac{b(\beta)\, \e^{-\beta/2}}{\e^{\beta/2}+\e^{-\beta/2}}\left[\frac{1}{\sinh\left((\gamma-\beta)/2\right)}-\frac{1}{\cosh\left((\gamma-\beta)/2\right)}\right]\frac{h(\gamma)}{b(\gamma)}\d\gamma,
\end{align*}
where $b(t):=(\e^{t}+\e^{-t})^{-1/2}(\e^{t/2}+\e^{-t/2})$ for $t\in\R$. Then, we get
\begin{align*}
\left[2\RR_0(\PV L_0) \RR_0^* h\right](\beta)&=\frac{ib(\beta)}{2\pi}\PV \int_{\R}\left[-\frac{1}{\sinh\left((\beta-\gamma)/2\right)}+\frac{\tanh(\beta/2)}{\cosh\left((\beta-\gamma)/2\right)}\right]\frac{h(\gamma)}{b(\gamma)}\d\gamma.
\end{align*}
Taking the Fourier transforms of hyperbolic functions (see \cite[Table 20.1]{J}) into account,  we have
\begin{equation}\label{expression for U2}
2\RR_0(\PV L_0) \RR_0^*=b(X)[-\tanh(\pi D)+i\tanh(X/2)\cosh(\pi D)^{-1}]b(X)^{-1}.
\end{equation}
This shows that the operator $2\RR_0\PV L_0\RR_0^*$ continuously extends to a bounded operator on $L^2(\R)$. Note that $\RR_0$ maps the subspace $C_c^{\infty}(-1,1)$ into the subspace $C_c^{\infty}(\R)$ bijectively.  Therefore, the equality $\RR U\RR^*=2\RR_0\PV L_0\RR_0^*$ follows from the equality \eqref{expression of U as singular integral} since $f$ and $g$ are arbitrary elements in $C_c^{\infty}(-1,1)$.  Now, one can easily observe that the commutator $[b(X),-\tanh(\pi D)+i\tanh(X/2)\cosh(\pi D)^{-1}]$ is compact since all the derivatives of the functions $b$, $\tanh$ and $\cosh(\cdot)^{-1}$ decay exponentially at infinity. Hence, the equality \eqref{expression for U} follows from \eqref{expression for U2}.
\end{proof}

\begin{proof}[Proof of Theorem \ref{main result 1}]
In view of \eqref{formula1} and Proposition \ref{prop1}, it suffices to prove that $K_0\F_{\sin}$ in the equality \eqref{formula1} is compact. Indeed, it follows from the estimate \eqref{estimate1} and from the equality $|s(\lambda)|=1$ that there exists $C'>0$ such that
\begin{equation}\label{estimate2}
|K_0(n,\lambda)|\leq C'(1-\lambda^2)^{-1/4}(1+n)^{-\rho+2}
\end{equation}
for any $\lambda\in(-1,1)$ and $n\in\Z_+$. Since the function $\lambda\mapsto(1-\lambda^2)^{-1/4}$ is square integrable on $(-1,1)$, $K_0\in\ell^2(\Z_+)\otimes L^2(-1,1)$ provided $\rho>5/2$. This implies that the kernel $(n,m)\mapsto [\F_{\sin}^* K_0(n,\cdot)](m)$ of the integral operator $K_0\F_{\sin}$ belongs to $\ell^2(\Z_+\times\Z_+)$. Hence, $K_0\F_{\sin}$ is a Hilbert-Schmidt operator and this finishes the proof. 
\end{proof}


\section{Topological version of Levinson's theorem}

A typical application of explicit formulas for the wave operators is the so-called \emph{topological version of Levinson's theorem} \cite{BS,I,KR,R}, which is an index theoretic interpretation of an equality of the type \eqref{classical Levinson}. Once an explicit formula for $W_-$ in terms of $(B,A)$ satisfying the Weyl relation \eqref{Weyl CCR} is established, the resulting index theorem follows automatically from an abstract $K$-theoretic argument for the short exact sequence
\begin{equation}\label{E}
0\to\K\xhookrightarrow{\iota} \E_{\square}\xrightarrowdbl{\pi} C(\square)\cong C(\S^1)\to0.
\end{equation}
Here $\K$ is the set of compact operators, $\E_{\square}$ is the $C^*$-subalgebra of $\B\left(\ell^2(\Z_+)\right)$ generated by the operators $a(A)b(B)$ with $a$ and $b$ belonging to the algebra $C([-\infty,\infty])$ of bounded continuous functions on $\R$ having limits at $\pm\infty$, and the space $\square$ is the boundary of $[-\infty,\infty]\times[-\infty,\infty]$.  Note that the map $\pi$ is given by $a(A)b(B)\mapsto \left(a(-\infty)b,b(-\infty)a,a(+\infty)b,b(+\infty)a\right)$.


In the present situation, one first observes that if we define a rescaled scattering matrix $S(\beta):=s(\tanh(\beta))$ for $\beta\in\R$, then $S\in C([-\infty,\infty])$, and ${\rm\bf S}=S(B)$. Hence, one deduces from formula \eqref{formula of wo} that the wave operator $W_-$ belongs to $\E_{\square}$ under the assumption \eqref{assumption}. One also easily checks that $\pi(W_-)=(S(\cdot),\Gamma_-,1,\Gamma_+)$ with
\begin{equation*}
\Gamma_{\pm}(\alpha)=1+\frac{s(\pm1)-1}{2}\left[1-\tanh(\pi \alpha)\pm i\cosh(\pi\alpha)^{-1}\right],\qquad \alpha\in\R,
\end{equation*}
and that $\pi(W_-)$ is a unitary element in $C(\square)$ by using Remark \ref{remark1}.  Hence, it follows from Atkinson's characterization that $W_-$ is a Fredholm operator on $\ell^2(\Z_+)$.


Now there are two homotopy invariants associated with $W_-$. One is the winding number $\wn(\pi(W_-))$ of the closed curve $\pi(W_-)(\square)\subset\C\setminus\{0\}$. As a convention, we turn around $\square$ clockwise. The other is the Fredholm index $\Index(W_-)$, and they are connected via


\begin{theorem}
Under the assumption \eqref{assumption}, the following index formula holds:
\begin{equation}\label{topological Levinson's theorem}
\wn\left(\pi(W_-)\right)=-\Index(W_-).
\end{equation}
\end{theorem}


We refer the reader to \cite[Prop. 7]{KR} or \cite[Sec. 4]{R} for purely $K$-theoretic arguments, and provide here a comparison of the relation  \eqref{classical Levinson} and the index formula \eqref{topological Levinson's theorem} as in \cite[Sec. 4]{I}. Note that it follows from the completeness and the Fredholmness of $W_-$ that $-\Index(W_-)$ is given by the number of eigenvalues of $H$, which we denoted by $N$ in Section \ref{Preliminaries}. On the other hand, the winding number comes from four contributions, one from each edge of $\square$. Since $S(\beta)=\e^{-2i\eta\left(\tanh(\beta)\right)}$, the contribution of $S(\cdot)$ is given by the phase shift $\eta$. One can also compute the contributions of $\Gamma_{\pm}$ by an inspection, and see that they equal $\pm\Delta_{\pm}$. According to the convention, $\wn\left(\pi(W_-)\right)$ is given by the signed sum
\begin{equation}
-\left\{\frac{-2\eta(+1)-\left(-2\eta(-1)\right)}{2\pi}\right\}+(-\Delta_-)+0-(+\Delta_+)=\frac{\eta(+1)-\eta(-1)}{\pi}-(\Delta_-+\Delta_+).
\end{equation}
Hence, \eqref{classical Levinson} can be obtained by rearranging the terms in the equality \eqref{topological Levinson's theorem}.


\begin{remark}\label{remark3}
Note that by using the addition formula for the sine function, one can easily prove that the shift operator $T$ satisfies $T=H_0+i(1-H_0^2)^{1/2}U^*$. Hence, it follows from the expression \eqref{expression for U} that
\begin{equation}\label{rescaled energy representation for T}
T=\tanh(B)-i\cosh(B)^{-1}\tanh(\pi A)+(\text{a compact operator}).
\end{equation}
Recall that the Toeplitz algebra $\TT$ is the $C^*$-subalgebra of $\B\left(\ell^2(\Z_+)\right)$ generated by the shift operator $T$, and that $\TT/\K\cong C(\T)$ with $\T=\{z\in\C\mid|z|=1\}$ \cite[Ex. 9.4.4]{RLL}.
 It follows from the equality \eqref{rescaled energy representation for T} that $\TT\subset\E_{\square}$ and that $\pi(T)=(\tanh(\cdot)+i\cosh(\cdot)^{-1},-1,\tanh(\cdot)-i\cosh(\cdot)^{-1},1)$. 
\end{remark}


\subsection*{Acknowledgements}
H. Inoue is supported by JSPS grant number JP18J21491. The authors gratefully acknowledge many helpful suggestions of professor S. Richard during the preparation of this paper. Especially, the decomposition of the operator $\RR_0\P L_0\RR_0^*$ in the proof of Proposition \ref{prop1} is due to him.


\end{document}